\documentclass[11pt,oneside,english,reqno]{amsart}
\usepackage[utopia]{mathdesign}
\usepackage[T1]{fontenc}
\usepackage[latin9]{inputenc}
\usepackage{url}
\usepackage{amsthm}

\makeatletter
\theoremstyle{plain}
  \theoremstyle{plain}
  \newtheorem*{thm*}{Theorem}
 \theoremstyle{definition}
 \newtheorem*{defn*}{Definition}
  \theoremstyle{plain}
  \newtheorem*{lem*}{Lemma}

\allowdisplaybreaks

\makeatother

\usepackage{babel}

\begin{document}

\title{The triangle and the open triangle}

\author{Gady Kozma}
\begin{abstract}
We show that for percolation on any transitive graph, the triangle
condition implies the open triangle condition.
\end{abstract}
\maketitle

\section{Introduction}

Let $G$ be a vertex-transitive%
\footnote{A vertex-transitive graph, and any other notion not specifically defined,
may be found in Wikipedia.%
} connected graph, and let $p$ be some number in $[0,1]$. We say
that $p$-percolation on $G$ satisfies the triangle condition if
for some $v\in G$ \begin{equation}
\sum_{x,y\in G}\mathbb{P}(v\leftrightarrow x)\mathbb{P}(x\leftrightarrow y)\mathbb{P}(y\leftrightarrow v)<\infty.\label{eq:triag}\end{equation}
where $x\leftrightarrow y$ implies that there exists an open path
between $x$ and $y$. Here and below we abuse notations by denoting
{}``$v$ is a vertex of $G$'' by $v\in G$. Of course, by transitivity,
the sum is in fact independent of $v$. This note is far too short
to explain the importance of the triangle condition. Suffices to say
that it the triangle condition holds at the \emph{critical }$p$,
then many exponents take their \emph{mean-field} values. See \cite{AN84,N87,BA91,KN09}
for corollaries of the triangle condition. On the other hand, the
triangle condition holds in many interesting cases, see \cite{HS90,HHS08}
for the graphs $\mathbb{Z}^{d}$ with $d$ sufficiently large, and
\cite{S01,S02,K} for various other transitive graphs. See \cite{G99}
or \cite{BR06} for a general introduction to percolation.

In many applications the triangle condition (\ref{eq:triag}) is not
so convenient to use. One instead uses the \emph{open }triangle condition,
which states that\[
\lim_{R\to\infty}\max_{w\not\in B(v,R)}\sum_{x,y\in G}\mathbb{P}(v\leftrightarrow x)\mathbb{P}(x\leftrightarrow y)\mathbb{P}(y\leftrightarrow w)=0,\]
where $B(v,R)$ stands for the ball around $v$ with radius $R$ in
the graph (or shortest path) distance. Clearly, the open triangle
condition implies the (closed) triangle condition (recall that if
$y$ and $y'$ are neighbors in the graph then $\mathbb{P}(x\leftrightarrow y)\ge c\mathbb{P}(x\leftrightarrow y')$
for some constant $c$ independent of $x$, $y$ and $y'$). The contents
of lemma 2.1 of Barsky \& Aizenman \cite{BA91} is the reverse implication.
The proof in \cite{BA91} is specific to the graph $\mathbb{Z}^{d}$
as it uses the Fourier transform of the function $f(x)=\mathbb{P}(\vec{0}\leftrightarrow x)$.
The purpose of this note is to generalize this to any transitive graph,
namely
\begin{thm*}
Let $G$ be a vertex-transitive graph and let $p\in[0,1]$. Assume
$G$ satisfies the triangle condition at $p$. Then $G$ satisfies
the open triangle condition at $p$.
\end{thm*}
This result is not particularly important. For example, in \cite{S01,S02}
the author simply circumvents the problem by working directly with
the open triangle condition. The advantage of making the triangle
condition {}``the'' marker for mean-field behavior is mostly aesthetic.
The real reason for the existance of this note is to demonstrate an
application of operator theory, specifically of spectral theory, to
percolation. Operator theory is a fantastically powerful tool whose
absence from the percolation scene is behind many of the difficulties
one encounters. I aim to remedy this situation, even if by very little.

I wish to thank Asaf Nachmias for pointing out some omissions in a
draft version of the paper, and Michael Aizenman for an intersting
discussion of alternative proof approaches.

\section{The proof}

Before starting the proof proper, let us make a short heuristic argument.
Define the infinite matrix

\begin{equation}
B(v,w)=\mathbb{P}(v\leftrightarrow w)\label{eq:defB}\end{equation}
where in the notation we assume that $v\leftrightarrow v$ always
so $B(v,v)=1$. By \cite{AN84} $B$, considered as an (unbounded)
operator on $l^{2}(G)$ is a positive operator. Hence the same holds
for

\begin{equation}
Q(v,w)=\sum_{x,y}B(v,x)B(x,y)B(y,w)\label{eq:defQ}\end{equation}
which is just $B^{3}$ (as an infinite matrix or as an unbounded operator).
It is possible to take the square root of any positive operator, so
denote $S=\sqrt{Q}$. We get\[
Q(v,w)=\langle Q\mathbf{1}_{v},\mathbf{1}_{w}\rangle=\langle S\mathbf{1}_{v},S\mathbf{1}_{w}\rangle\]
where $\mathbf{1}_{v}$ is the element of $l^{2}(G)$ defined by \[
\mathbf{1}_{v}(x)=\begin{cases}
1 & v=x\\
0 & v\ne x.\end{cases}\]
Hence the triangle condition $Q(v,v)<\infty$ implies that $||Sv||<\infty$.
But $S$ is invariant to the automorphisms of $G$ (as a root of $Q$
which is invariant to them) so $S\mathbf{1}_{w}$ is a map of $S\mathbf{1}_{v}$
under an automorphism $\varphi$ taking $v$ to $w$. But any vector
in $l^{2}$ is almost orthogonal to sufficiently far away {}``translations''
(namely, the automorphisms of $G$), so $\langle S\mathbf{1}_{v},S\mathbf{1}_{w}\rangle\to0$
as the graph distance of $v$ and $w$ goes to $\infty$, as required.

Why is this even a heuristic and not a full proof? Because of the
benign looking expression $\langle Q\mathbf{1}_{v},\mathbf{1}_{w}\rangle$
which is in fact meaningless. $Q$ is an unbounded operator and hence
it cannot be applied to any vector in $l^{2}(G)$, and there is nothing
guaranteeing that $\mathbf{1}_{v}$ will be in its domain. For example,
in a sufficiently spread-out lattice in $\mathbb{R}^{d}$ one has
that $\mathbb{P}(x\leftrightarrow y)\approx|x-y|^{2-d}$ \cite{HHS03}
which gives with a simple calculation that the triangle condition
holds whenever $d>6$ while $Q\mathbf{1}_{v}\in l^{2}$ only when
$d>12$.

The proof below circumvents this problem by decomposing $B$ into
a sum of positive bounded operators using specific properties of $B$.
Somebody more versed in the theory of unbounded operators might have
constructed a more direct proof. 

We start the proof proper with
\begin{defn*}
Let $\varphi$ be an automorphism of the graph $G$. We define the
isometry $\Phi=\Phi_{\varphi}$ of $l^{2}(G)$ corresponding to $\varphi$
by\begin{equation}
(\Phi(f))(v)=f(\varphi^{-1}(v)).\label{eq:defPhi}\end{equation}

\end{defn*}
It is easy to check that $\Phi\mathbf{1}_{v}=\mathbf{1}_{\varphi(v)}$
and that the support of $\Phi f$ is $\varphi($the support of $f)$.
\begin{lem*}
Let $f\in l^{2}(G)$, let $v\in G$ and let $\delta>0$. Then there
exists an $R=R(f,\delta,v)$ such that for any $w\not\in B(v,R)$
and any automorphism $\varphi$ of $G$ taking $v$ to $w$ one has\begin{equation}
|\langle\Phi_{\varphi}f,f\rangle|<\delta\label{eq:almost}\end{equation}
\end{lem*}
\begin{proof}
Let $A\subset G$ be some finite set of vertices such that\[
\sqrt{\sum_{v\not\in A}|f(v)|^{2}}<\frac{1}{3||f||}\delta.\]
Write now\[
f=f_{\mathrm{loc}}+f_{\mathrm{glob}}\mbox{ where }f_{\mathrm{loc}}=f\cdot\mathbf{1}_{A}.\]
By the definition of $A$, $||f_{\mathrm{glob}}||<\frac{1}{3||f||}\delta$,
and so by Cauchy-Schwarz,\begin{equation}
|\langle\Phi f,f\rangle|\le|\langle\Phi f_{\mathrm{loc}},f_{\mathrm{loc}}\rangle|+2||f_{\mathrm{glob}}||\cdot||f_{\mathrm{loc}}||+||f_{\mathrm{glob}}||^{2}<|\langle\Phi f_{\mathrm{loc}},f_{\mathrm{loc}}\rangle|+\delta.\label{eq:Phifloc}\end{equation}
Define now \[
R=2\max_{x\in A}d(v,x)+1.\]
To see (\ref{eq:almost}), let $w$ and $\varphi$ be as above. We
get, for any $x\in A$, \[
d(\varphi(x),v)\ge d(v,w)-d(\varphi(x),w).\]
Now, $d(\varphi(x),w)=d(\varphi(x),\varphi(v))=d(x,v)<\frac{1}{2}R$
because $\varphi$ is an automorphism of $G$. Hence we get\[
d(\varphi(x),v)>R-{\textstyle \frac{1}{2}}R\]
implying that $\varphi(x)\not\in A$ as it is too far. In other words,
$A\cap\varphi(A)=\emptyset$ which implies that $\langle\Phi_{\varphi}f_{\mathrm{loc}},f_{\mathrm{loc}}\rangle=0$.
With (\ref{eq:Phifloc}), the lemma is proved.
\end{proof}

\begin{proof}
[Proof of the theorem]We will not keep $p$ in the notations as it
does not change throughout the proof. For every $n\in\mathbb{N}$
and every $v,w\in G$, let $B_{n}(v,w)$ be defined by\[
B_{n}(v,w)=\mathbb{P}(v\leftrightarrow w,\,|\mathcal{C}(v)|=n)\]
where $\mathcal{C}(v)$ is the cluster of $v$ i.e.\ the set of vertices
connected to $v$ by open paths, and $|\mathcal{C}(v)|$ is the number
of vertices in $\mathcal{C}(v)$. Clearly $B_{n}(v,w)\ge0$ and \begin{equation}
B(v,w)=\sum_{n=1}^{\infty}B_{n}(v,w)\label{eq:BBn}\end{equation}
where $B$ is as above (\ref{eq:defB}). Therefore we may write\begin{align}
Q(v,w) & \stackrel{(\ref{eq:defQ})}{=}\sum_{x,y}B(v,x)B(x,y)B(y,w)\stackrel{(\ref{eq:BBn})}{=}\sum_{x,y}B(v,x)\left(\sum_{n=1}^{\infty}B_{n}(x,y)\right)B(y,w)=\nonumber \\
 & \stackrel{\hphantom{(\ref{eq:defQ})}}{=}\sum_{n=1}^{\infty}\sum_{x,y}B(v,x)B_{n}(x,y)B(y,w)\label{eq:QsumBBnB}\end{align}
where the change of order of summation in the last equality is justified
since all terms are positive. Now, the vector \[
B\mathbf{1}_{w}=\left(B(y,w)\right)_{y\in G}\]
is in $l^{2}(G)$ because\[
\sum_{y}B(y,w)^{2}\le\sum_{y,x}B(w,y)B(y,x)B(x,w)<\infty.\]
Further, each $B_{n}$, considered as an operator on $l^{2}(G)$ is
bounded, because the sum of the (absolute values of the) entries in
each row and each column is finite. From this we conclude that $B_{n}B\mathbf{1}_{w}\in l^{2}(G)$
and we may present the sum in (\ref{eq:QsumBBnB}) in an $l^{2}$
notation as\begin{equation}
Q(v,w)=\sum_{n=1}^{\infty}\langle B_{n}B\mathbf{1}_{v},B\mathbf{1}_{w}\rangle.\label{eq:QBnB,B}\end{equation}
Next we employ the argument of Aizenman \& Newman \cite{AN84} to
show that $B_{n}$ is a positive operator. This means that $B_{n}(v,w)=B_{n}(w,v)$
(which is obvious) and that $\langle B_{n}f,f\rangle\ge0$ for any
(real-valued) $f\in l^{2}$. It is enough to verify this for $f$
with finite support. But in this case we can write\begin{align*}
\langle B_{n}f,f\rangle & =\sum_{v,w}f(v)f(w)\mathbb{P}(v\leftrightarrow w,\,|\mathcal{C}(v)|=n)=\\
(*)\qquad & =\mathbb{E}\Big(\sum_{v,w}f(v)f(w)\mathbf{1}_{\{v\leftrightarrow w,|\mathcal{C}(v)|=n\}}\Big)=\\
 & =\mathbb{E}\Big(\sum_{\mathcal{C}\:\mathrm{s.t.}\:|\mathcal{C}|=n}\;\sum_{v,w\in\mathcal{C}}f(v)f(w)\Big)=\mathbb{E}\Big(\sum_{\mathcal{C}\:\mathrm{s.t.}\:|\mathcal{C}|=n}\Big(\sum_{v\in\mathcal{C}}f(v)\Big)^{2}\Big)\ge0.\end{align*}
where $(*)$ is where we used the fact that $f$ has finite support
to justify taking the expectation out of the sum. The notation $\mathbf{1}_{E}$
here is for the indicator of the event $E$. Thus $B_{n}$ is positive.

We now apply the spectral theorem for \emph{bounded }positive operators
to take the square root of $B_{n}$. See \cite{EMT04}, lemma 6.3.5
for the specific case of taking the root of a positive operator and
chapter 7 for general spectral theory. Denote $S_{n}=\sqrt{B_{n}}$.
This implies, of course, that $S_{n}^{2}=B_{n}$ but also that $S_{n}$
is positive and that it commutes with any operator $\Phi$ that commutes
with $B_{n}$.

Returning to (\ref{eq:QBnB,B}) we now write\begin{equation}
Q(v,w)=\sum_{n=1}^{\infty}\langle S_{n}^{2}B\mathbf{1}_{v},B\mathbf{1}_{w}\rangle=\sum_{n=1}^{\infty}\langle S_{n}B\mathbf{1}_{v},S_{n}B\mathbf{1}_{w}\rangle.\label{eq:QSBSB}\end{equation}
The fact that $Q(v,v)<\infty$ therefore implies that\begin{equation}
\sum_{n=1}^{\infty}||S_{n}B\mathbf{1}_{v}||^{2}<\infty.\label{eq:sumSnB}\end{equation}
Our only use of the triangle condition.

Fix now some $\epsilon>0$. By (\ref{eq:sumSnB}) we can find some
$N$ such that \begin{equation}
\sum_{n=N+1}^{\infty}||S_{n}B\mathbf{1}_{v}||^{2}<\tfrac{1}{2}\epsilon.\label{eq:N+1inf}\end{equation}
Since $S_{n}B\mathbf{1}_{v}\in l^{2}(G)$, we can use the lemma, and
we use it with \[
f_{\mathrm{lemma}}=S_{n}B\mathbf{1}_{v}\quad v_{\mathrm{lemma}}=v\qquad\delta_{\mathrm{lemma}}=\frac{\epsilon}{2N}\:.\]
We get some $R_{n}$ such that for any $\varphi$ taking $v$ outside
of $B(v,R_{n})$, \[
|\langle\Phi_{\varphi}S_{n}B\mathbf{1}_{v},S_{n}B\mathbf{1}_{v}\rangle|\le\frac{\epsilon}{2N}\:.\]
Some standard abstract nonsense shows that the invariance of $B_{n}$
i.e.\ the fact that $B_{n}(x,y)=B_{n}(\varphi(x),\varphi(y))$ implies
that $B_{n}\Phi=\Phi B_{n}$. Hence also $S_{n}\Phi=\Phi S_{n}$ so
\[
\langle\Phi S_{n}B\mathbf{1}_{v},S_{n}B\mathbf{1}_{v}\rangle=\langle S_{n}B\Phi\mathbf{1}_{v},S_{n}B\mathbf{1}_{v}\rangle=\langle S_{n}B\mathbf{1}_{\varphi(v)},S_{n}B\mathbf{1}_{v}\rangle.\]
Define $R=\max\{R_{1},\dotsc,R_{N}\}$. We get, for every $w\not\in B(v,R)$,
\begin{equation}
\sum_{n=1}^{N}\langle S_{n}B\mathbf{1}_{v},S_{n}B\mathbf{1}_{w}\rangle\le N\delta=\tfrac{1}{2}\epsilon.\label{eq:sum1N}\end{equation}
(\ref{eq:N+1inf}) takes care of the other sum,\begin{align}
\sum_{n=N+1}^{\infty}\langle S_{n}B\mathbf{1}_{v},S_{n}B\mathbf{1}_{w}\rangle & \leq\sum_{n=N+1}^{\infty}||S_{n}B\mathbf{1}_{v}||\cdot||S_{n}B\mathbf{1}_{w}||=\nonumber \\
 & =\sum_{n=N+1}^{\infty}||S_{n}B\mathbf{1}_{v}||^{2}<\tfrac{1}{2}\epsilon.\label{eq:sumNinf}\end{align}
We are done. We get that for any $w\not\in B(v,R)$,\[
Q(v,w)\stackrel{(\ref{eq:QSBSB})}{=}\sum_{n=1}^{\infty}\langle S_{n}B\mathbf{1}_{v},S_{n}B\mathbf{1}_{w}\rangle\stackrel{(\ref{eq:sum1N},\ref{eq:sumNinf})}{\le}\epsilon\]
as required.
\end{proof}
\noindent \emph{Closing remark.} Comparing the proof here to that
of Barsky \& Aizenman \cite{BA91}, it seems as if there is something
missing in their argument. This is not true. Justifying the change
of order of summation in \cite{BA91} is completely standard --- for
example, by examining Cesàro sums --- and does not deserve any special
remark.

\end{document}